\documentclass[11pt]{article}
\usepackage{epsf,amsmath,amsfonts}

\title{The excluded minor structure theorem with planarly embedded wall}

\author{Bojan Mohar\thanks{Supported in part by ARRS,  
   Research Program P1-0297, by an NSERC Discovery Grant and
   by the Canada Research Chair program.
   On leave from Department of Mathematics,
   University of Ljubljana, Ljubljana, Slovenia.
   Email address: {\tt mohar@sfu.ca}}\\
   Department of Mathematics\\
   Simon Fraser University\\ 
   Burnaby, B.C. V5A 1S6, Canada
}

\newcommand{\DEF}[1]{{\em #1\/}}

\newtheorem{theorem}{\bf Theorem}[section]

\newtheorem{lemma}[theorem]{\bf Lemma}

\def\slika #1{\begin{center}\hskip 0.1mm\epsffile{#1}\end{center}}
\newenvironment{proof}%
{\noindent{\bf Proof.}\ }%
{\hfill$\Box$\par\bigskip}%

\renewcommand\a{\alpha}

\begin{document}

\date{\today}
\maketitle

\begin{abstract}
A graph is ``nearly embedded'' in a surface if it consists of graph $G_0$ that is
embedded in the surface, together with a bounded number of vortices having no 
large transactions. It is shown that every large wall (or grid minor) in a nearly
embedded graph, many rows of which intersect the embedded subgraph $G_0$ of the 
near-embedding, contains a large subwall that is planarly embedded within $G_0$.
This result provides some hidden details needed for a strong version of
the Robertson and Seymour's excluded minor theorem as presented in \cite{BKMM}.
\end{abstract}

%%%%%%%%%%%%%%%%%%%%%%%%%%%%%%%%%%%%%%%%%%%%%%%%%%%%%%%%

\section{Introduction}

A graph is a \DEF{minor} of another graph if
the first can be obtained from a subgraph of the second by
contracting edges. One of the highlights of the 
graph minors theory developed by Robertson and Seymour is the
Excluded Minor Theorem (EMT) that describes a rough structure of graphs 
that do not contain a fixed graph $H$ as a minor. Two versions of
EMT appear in \cite{RS16,RS17}; see also \cite{Di} and \cite{KMSurvey}.

In \cite{BKMM} and \cite{BKMM3} the authors used a strong version of 
EMT in which it is concluded that every graph without a fixed minor
and whose tree-width is large has a tree-like structure, whose pieces are
subgraphs that are almost embedded in some surface, and 
in which one of the pieces contains a large grid minor that is (essentially) 
embedded in a disk on the surface.
Although not explicitly mentioned, this version of EMT
follows from the published results of Robertson and Seymour 
\cite{RS17} by applying standard techniques of routings on surfaces.
Experts in this area are familiar with these techniques (that are also
present in Robertson and Seymour's work \cite{RS7}). However, 
they may be harder to digest for newcomers in the area, and thus
deserve to be presented in the written form. 
The purpose of this note is to provide a proof of an extended
version of EMT as stated in \cite[Theorem 4.2]{BKMM}.

It may be worth mentioning that the proof in \cite{BKMM} does not really need
the extended version of the EMT, but the proof in \cite{BKMM3} does.
Thus, this note may also be viewed as a support for the main proof in \cite{BKMM3}.

We assume that the reader is familiar with the basic notions of graph theory and
in particular with the basic notions related to graph minors; we refer to 
\cite{Di} for all terms and results not explained here.

%---------------------------------
\section{Walls in near-embeddings}
%---------------------------------

In this section, we present our main lemma, which shows that 
for every large wall (to be defined in the sequel) in a ``nearly embedded''
graph, a large subwall must be contained in the embedded subgraph of the
near-embedding. Let us first introduce the notion of the wall and some of its
elementary properties.

\begin{figure}[htb]
\epsfxsize=5.6truecm
\slika{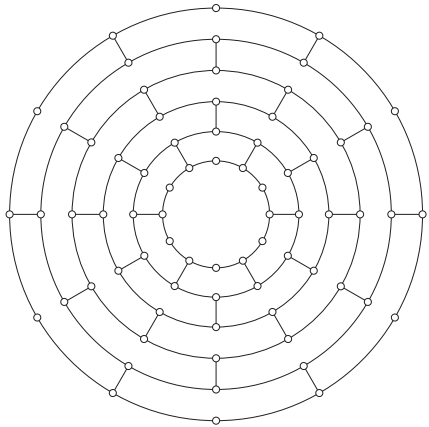}
\caption{The cylindrical 6-wall $Q_6$}
\label{fig:1}
\end{figure}

For an integer $r\ge3$, we define a \DEF{cylindrical $r$-wall} as a graph 
that is isomorphic to a subdivision of the graph $Q_r$ defined as follows. 
We start with vertex set $V = \{(i,j) \mid 1\le i \le r, \, 1\le j \le 2r \}$,
and make two vertices $(i,j)$ and $(i',j')$ adjacent
if and only if one of the following possibilities holds:
\begin{itemize}
\item[(1)] $i' = i$ and $j' \in \{j-1,j+1\}$, where the values $j-1$ and $j+1$
are considered modulo $2r$.
\item[(2)] $j' = j$ and $i' = i + (-1)^{i+j}$.
\end{itemize}
Less formally, $Q_r$ consists of $r$ disjoint cycles $C_1,\dots,C_r$ of length 
$2r$ (where $V(C_i)=\{(i,j)\mid 1\le j\le 2r\}$), called the \DEF{meridian cycles} 
of $Q_r$. Any two consecutive cycles $C_i$ and $C_{i+1}$ are joined by $r$ edges
so that the edges joining $C_i$ and $C_{i-1}$ interlace on $C_i$ with those
joining $C_i$ and $C_{i+1}$ for $1<i<r$.
Figure \ref{fig:1} shows the cylindrical 6-wall $Q_6$.

By deleting the edges joining vertices $(i,1)$ and $(i,2r)$ for $i=1,\dots,r$, we
obtain a subgraph of $Q_r$. Any graph isomorphic to a subdivision of this graph is
called an \DEF{$r$-wall}.

To relate walls and cylindrical walls to $(r\times r)$-grid minors, we state 
the following easy correspondence:

\begin{itemize}
\item[\rm (a)]
Every $(4r+2)$-wall contains a cylindrical $r$-wall as a subgraph.
\item[\rm (b)]
Every cylindrical $r$-wall contains an $(r\times r)$-grid as a minor.
\item[\rm (c)]
Every $(r\times r)$-grid minor contains an $\lfloor \tfrac{r-1}{2}\rfloor$-wall
as a subgraph.
\end{itemize}

\begin{lemma}
\label{lem:W1}
Suppose that\/ $1\le i<j\le r$ and let $t=j-i-1$. Let $S_i\subset C_i$ and 
$S_j\subseteq C_j$ be paths of length at least $2t-1$ in the meridian cycles 
$C_i,C_j$ of $Q_r$. Then\/ $Q_r$ contains $t$ disjoint paths linking
$C_i$ and $C_j$. Morover, for each of these paths and for every cycle $C_k$,
$i<k<j$, the intersection of the path with $C_k$ is a connected segment of $C_k$.
\end{lemma}

\begin{figure}[htb]
\epsfxsize=7.5truecm
\slika{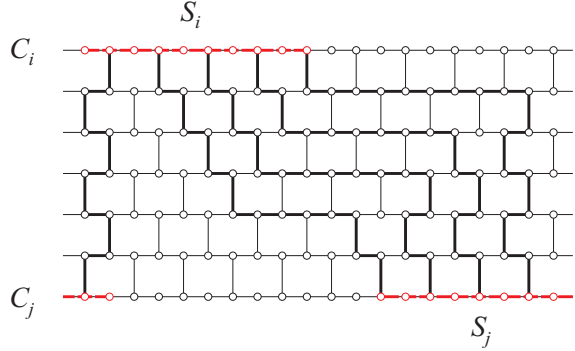}
\caption{Paths linking $S_i$ and $S_j$}
\label{fig:5}
\end{figure}

\begin{proof}
The lemma is easy to prove and the idea is illustrated in Figure~\ref{fig:5},
in which the edges on the left are assumed to be identified with the 
corresponding edges on the right. The paths are shown by thick lines and the 
segments $S_i$ and $S_j$ are shown by thick broken lines.
\end{proof}

A \DEF{surface} is a compact connected 2-manifold (with or without
boundary). The components of the boundary are called the \DEF{cuffs}.
If a surface $S$ has Euler characteristic $c$, then the non-negative number
$g=2-c$ is called the \DEF{Euler genus} of $S$. Note that a surface of
Euler genus $g$ contains at most $g$ cuffs.

Disjoint cycles $C,C'$ in a graph embedded in a surface $S$ are 
\DEF{homotopic} if there is a cylinder in $S$ whose boundary components are
the cycles $C$ and $C'$. The cylinder bounded by homotopic cycles $C,C'$ 
is denoted by $int(C,C')$.
Disjoint paths $P,Q$ whose initial vertices lie in the same cuff $C$ and
whose terminal vertices lie in the same cuff $C'$ in $S$ (possibly $C'=C$)
are \DEF{homotopic} if $P$ and $Q$ together with a segment in $C$ and a segment 
in $C'$ form a contractible closed curve $A$ in $S$.
The disk bounded by $A$ is be denoted by $int(P,Q)$.
The following basic fact about homotopic curves on a surface will be used
throughout (cf., e.g., \cite[Propositions 4.2.6 and 4.2.7]{MT}).

\begin{lemma}
\label{lem:homotopic}
Let $S$ be a surface of Euler genus $g$. Then every collection of more than
$3g$ disjoint non-contractible cycles contains two cycles that are homotopic.
Similarly, every collection of more than $3g$ disjoint paths, whose ends are 
on the same (pair of) cuffs in $S$, contains two paths that are homotopic.
\end{lemma}

Let $G$ be a graph and let $W=\{w_0,\dots, w_n\}$, $n=|W|-1$, be a
linearly ordered subset of its vertices such that $w_i$ precedes
$w_j$ in the linear order if and only if $i<j$.  The pair $(G,W)$ is
called a \DEF{vortex} of \DEF{length} $n$, $W$ is the \DEF{society}
of the vortex and all vertices in $W$ are called \DEF{society vertices}.
When an explicit reference to the society is not needed, we will as well
say that $G$ is a vortex.
A collection of disjoint paths $R_1,\dots,R_k$ in $G$ is called 
a \DEF{transaction} of \DEF{order} $k$ in the vortex $(G,W)$ if 
there exist $i,j$ ($0\le i\le j\le n$) such that all paths have their 
initial vertices in $\{w_i,w_{i+1},\dots, w_j\}$ and their endvertices
in $W\setminus \{w_i,w_{i+1},\dots, w_j\}$. 

Let $G$ be a graph that can be expressed as $G=G_0\cup G_1\cup\cdots\cup G_v$,
where $G_0$ is embedded in a surface $S$ of Euler genus $g$ with $v$ cuffs
$\Omega_1,\dots, \Omega_v$, and $G_i$ ($i=1,\dots,v$) are pairwise disjoint 
vortices, whose society is equal to their intersection with $G_0$ and is 
contained in the cuff $\Omega_i$, with the order of the society being inherited 
from the circular order around the cuff. Then we say that $G$ is \DEF{near-embedded}
in the surface $S$ \DEF{with vortices} $G_1,\dots,G_v$.
A subgraph $H$ of a graph $G$ that is near-embedded in $S$ is said to be 
\DEF{planarly embedded} in $S$ if $H$ is contained in the embedded subgraph 
$G_0$, and there exists a cycle $C\subseteq G_0$ that is contractible in $S$ 
and $H$ is contained in the disk on $S$ that is bounded by $C$.
Our main result is the following.

\begin{theorem}
\label{th:main}
For every non-negative integers $g,v,a$ there exists a positive integer 
$s=s(g,v,a)$ such that the following holds.
Suppose that a graph $G$ is near-embedded in the surface $S$ with vortices 
$G_1,\dots,G_v$, such that the maximum order of transactions of
the vortices is at most $a$. Let $Q$ be a cylindrical $r$-wall contained in $G$,
such that at least $r_0\ge 3s$ of its meridian cycles have at least one edge
contained in $G_0$. Then $Q\cap G_0$ contains a cylindrical $r'$-wall   
that is planarly embedded in $S$ and has $r'\ge r_0/s$.
\end{theorem}

\begin{proof}
Let $C_{p_1},C_{p_2},\dots,C_{p_{r_0}}$ ($p_1< p_2<\cdots <p_{r_0}$) be meridian
cycles of $Q$ having an edge in $G_0$. For $i=1,\dots,r_0$, let $L_i$ be a maximal
segment of $C_{p_i}$ containing an edge in $E(C_{p_i})\cap E(G_0)$ and such that
none of its vertices except possibly the first and the last vertex are on a cuff. 
It may be that $L_i=C_{p_i}$ if $C_{p_i}$ contains at most one vertex on a cuff;
if not, then $L_i$ starts on some cuff and ends on (another or the same) cuff.
(We think of the meridian cycles to have
the orientation as given by the meridians in the wall.) 
At least $r_0/(v^2+1)$ of the segments $L_i$ either start and end up on the same 
cuffs $\Omega_x$ and $\Omega_y$ (possibly $x=y$), or are all cycles. 
In each case, we consider their homotopies. 
By Lemma \ref{lem:homotopic}, these segments contain a subset of 
$q \ge r_0/((3g+1)(v^2+1))$ homotopic segments (or cycles).
Since we will only be interested in these homotopic segments or cycles, we will
assume henceforth that $L_1,\dots,L_q$ are homotopic.

\begin{figure}[htb]
\epsfxsize=6.8truecm
\slika{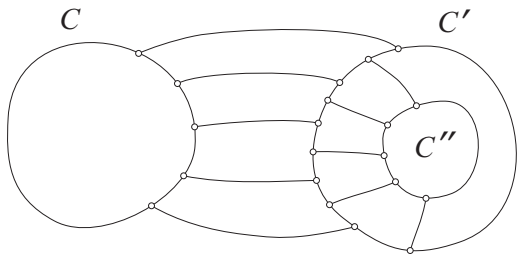}
\caption{Many contractible cycles}
\label{fig:4}
\end{figure}

Let us first look at the case when $L_1,\dots,L_q$ are cycles. Since $s=s(g,v,a)$
can be chosen to be arbitrarily large (as long as it only depends on the parameters),
we may assume that $q$ is as large as needed in the sequel. 
If the cycles $L_i$ are pairwise homotopic and non-contractible, then it is easy to see
that two of them bound a cylinder in $S$ containing many of these cycles. This cylinder
also contains the paths connecting these cycles; thus it contains a large planarly
embedded wall and hence also a large planarly embedded cylindrical wall. 
So, we may assume that the cycles $L_1,\dots,L_q$ are contractible.
By Lemma \ref{lem:W1}, $Q$ contains $t$ paths linking any two of these cycles that
are $t$ apart in $Q$, say $C=L_i$ and $C'=L_{i+t+1}$. (Here we take $t$ large enough
that the subsequent arguments will work.) Again, many of these paths either
reach $C'$ without intersecting any of the cuffs, or many reach 
the same cuff $\Omega$. 
A large subset of them is homotopic. In the former case,
the paths linking $C'$ with $C''=L_{i+2t+2}$ can be chosen so that their
initial vertices interlace on $C'$ with the end-vertices of of the
homotopic paths coming from $C$. This implies that $C$ or
$C''$ lies in the disk bounded by $C'$ (cf.\ Figure~\ref{fig:4}). By repeating 
the argument, we obtain a sequence of nested cycles and interlaced linkages 
between them. This clearly gives a large subwall, which contains a large 
cylindrical subwall that is planarly embedded. In the latter case, 
when the paths from $C$ to $C'$ go through the same cuff $\Omega_j$,
we get a contradiction since the vortex on $\Omega_j$ does not admit 
a transaction of large order, and thus too many homotopic paths cannot reach 
$C''$.

\begin{figure}[htb]
\epsfxsize=10truecm
\slika{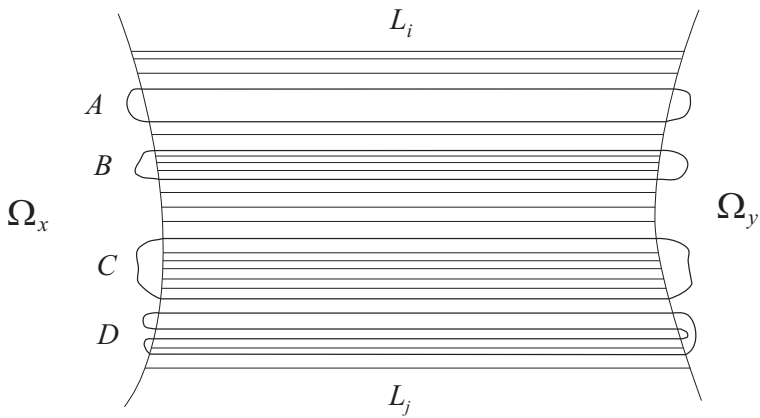}
\caption{Many homotopic segments joining two cuffs}
\label{fig:3}
\end{figure}

We get a similar contradiction as in the last case above, when too many homotopic
segments $L_i$ start and end up on the same cuffs $\Omega_x$ and $\Omega_y$.
We shall give details for the case when $x\ne y$, but the same approach
works also if $x=y$. (In the case when $x=y$ and the homotopic segments $L_i$ are
contractible, the proof is similar to the part of the proof given above.)

Let us consider the ``extreme'' segments $L_i, L_j$, whose disk
$int(L_i,L_j)$ contains many homotopic segments (cf.\ Figure~\ref{fig:3}).
Let us enumerate these segments as $L'_1=L_i,L'_2,\dots, L'_m=L_j$ in the order as
they appear inside $int(L_i,L_j)$. Let $C'_t$ (for $1<t<m$) be the meridian
cycle containing the segment $L'_t$. Since vortices admit no transactions of
order more than $a$, at most $4a$ of the cycles $C'_t$ ($1<t<m$)
can leave $int(L_i,L_j)$. By adjusting $m$, we may thus assume that none of them
does. In particular, each $L'_t$ has another homotopic segment in
$int(L_i,L_j)$. Since there are no transactions of order more than $a$, there is a
large subset of the cycles $C'_t$ that follow each other in $int(L_i,L_j)$
as shown by the thick cycles in Figure~\ref{fig:3}. Consider four of these
meridian cycles $A,B,C,D$ that are pairwise far apart in the wall $Q$ and 
appear in the order $A,B,C,D$ within $int(L_i,L_j)$. Then $A$ and $C$ are linked 
in $Q$ by a large collection of disjoint paths by Lemma \ref{lem:W1}. 
At most $8a$ of these paths can escape intersecting two fixed segments 
$L'_u$ and $L'_v$ of $B$ or two such segments of $D$ by passing through a vortex. 
All other paths linking $A$ and $C$ intersect either two segments of $B$ or two 
segments of $D$.
However, this is a contradiction since the paths linking $A$ and $C$ can be
chosen in $Q$ so that each of them intersects each meridian cycle in a connected
segment (Lemma~\ref{lem:W1}). This completes the proof.
\end{proof}

%-------------------------------------
\section{The excluded minor structure}
%-------------------------------------

In this section, we define some of the structures found in
Robertson-Seymour's Excluded Minor Theorem \cite{RS16} which
describes the structure of graphs
that do no contain a given graph as a minor. Robertson and Seymour
proved a strengthened version of that theorem that gives a more
elaborate description of the structure in \cite{RS17}. 
Our terminology follows that introduced in \cite{BKMM}.

Let $G_0$ be a graph. Suppose that $(G_1',G_2')$ is
a separation of $G_0$ of order $t\le 3$, i.e., $G_0 = G_1'\cup G_2'$,
where $G_1'\cap G_2' = \{v_1,\dots,v_t\}\subset
V(G_0)$, $1\le t\le 3$, $V(G_2')\setminus V(G_1')\ne\emptyset$.
Let us replace $G_0$ by the graph $G'$, which is obtained from
$G_1'$ by adding all edges $v_iv_j$ ($1\le i < j \le t$) if they are
not already contained in $G_1'$. We say that $G'$ has been obtained
from $G_0$ by an \DEF{elementary reduction}.
If $t=3$, then the 3-cycle $T=v_1v_2v_3$ in $G'$ is called
the \DEF{reduction triangle}. Every graph $G''$ that can be obtained
from $G_0$ by a sequence of elementary reductions is a \DEF{reduction}
of $G_0$.

We say that a graph $G_0$ can be \DEF{embedded} in a surface $\Sigma$ 
\DEF{up to $3$-separations} if there is a reduction $G''$ of $G_0$
such that $G''$ has an embedding in $\Sigma$ in which every
reduction triangle bounds a face of length 3 in $\Sigma$.

Let $H$ be an $r$-wall in the graph $G_0$ and let $G''$ be
a reduction of $G_0$. We say that the reduction $G''$ \DEF{preserves} $H$ 
if for every elementary reduction used in obtaining $G''$ from $G_0$, at most
one vertex of degree 3 in $H$ is deleted. (With the above notation,
$G_2'\setminus G_1'$ contains at most one vertex of degree 3 in $H$.)

\begin{lemma}
\label{lem:captured wall}
Suppose that\/ $G''$ is a reduction of the the graph $G_0$ and that\/
$G''$ preserves an $r$-wall $H$ in $G_0$. Then $G''$ contains
an $\lfloor (r+1)/3\rfloor$-wall, all of whose edges are contained
in the union of $H$ and all edges added to $G''$ when performing
elementary reductions.
\end{lemma}

\begin{figure}[htb]
\epsfxsize=4truecm
\slika{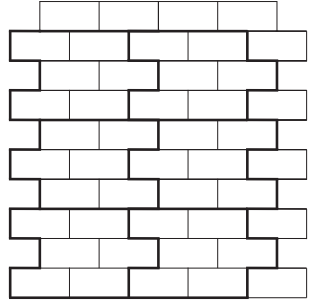}
\caption{Smaller wall contained in a bigger wall}
\label{fig:2}
\end{figure}

\begin{proof}
Let $H'$ be the subgraph of the $r$-wall $H$ obtained by taking every
third row and every third ``column''. See Figure \ref{fig:2} in which
$H'$ is drawn with thick edges.
It is easy to see that for every elementary reduction
we can keep a subgraph homeomorphic to $H'$ by replacing the edges
of $H'$ which may have been deleted by adding some of the edges $v_iv_j$
involved in the reduction.
The only problem would occur when we lose a vertex of degree 3 and when
all vertices $v_1,v_2,v_3$ involved in the elementary reduction would be
of degree 3 in $H'$. However, this is not possible since $G''$ preserves $H$.
\end{proof}

Suppose that for $i=0,\dots,n$, there exist vertex sets,
called \DEF{parts}, $X_i \subseteq V(G)$, with the following
properties:
\begin{itemize}
\item[(V1)] $X_i \cap W = \{w_i,w_{i+1}  \}$ for $i=0,\dots,n$, where
$w_{n+1}=w_n$,
\item[(V2)] $\bigcup_{0\le i\le n} X_i = V(G)$,
\item[(V3)] every edge of $G$ has both endvertices in some $X_i$, and
\item[(V4)] if $i \leq j \leq k$, then $X_i\cap X_k\subseteq X_j$.
\end{itemize}

\noindent Then the family $(X_i \, ;\, i=0,\dots,n)$ is called
a \DEF{vortex decomposition} of the vortex $(G,W)$.
%The \DEF{width} of the vortex decomposition is the maximum
%of $\{|X_i| \, ;\, i=0,\dots,n\}$.
For $i=1,\dots,n$, denote by $Z_i = (X_{i-1}\cap X_i)\setminus W$.
The \DEF{adhesion} of the vortex decomposition is the maximum of
$|Z_i|$, for $i=1,\dots,n$. The vortex decomposition is \DEF{linked}
if for $i=1,\dots,n-1$, the subgraph of $G$ induced on the vertex
set $X_i\setminus W$ contains a collection of disjoint paths linking
$Z_i$ with $Z_{i+1}$. Clearly, in that case $|Z_i|=|Z_{i+1}|$, and
the paths corresponding to $Z_i\cap Z_{i+1}$ are trivial.
Note that (V1) and (V3) imply that there are no edges between
nonconsecutive society vertices of the vortex.
Let us remark that every vortex $(G,W)$, in which $w_i,w_j$ are
non-adjacent for $|i-j|\ge2$, admits a linked vortex decomposition;
just take $X_i = (V(G)\setminus W) \cup \{w_i,w_{i+1} \}$.

The \DEF{$($linked$)$ adhesion of the vortex}
is the minimum adhesion taken over all (linked) decompositions of the
vortex. Let us observe that in a linked decomposition of adhesion
$q$, there are $q$ disjoint paths linking $Z_1$ with $Z_n$ in $G-W$.
For us it is important to note that a vortex with adhesion less than $k$
does not admit a transaction of order more than $k$.

Let $G$ be a graph, $H$ an $r$-wall in $G$, $\Sigma$ a surface, and
$\alpha\ge0$ an integer. We say that $G$ can be
\DEF{$\alpha$-nearly embedded} in $\Sigma$ if there is a set of at
most $\alpha$ cuffs $C_1,\dots ,C_b$ ($b\le \alpha$) in $\Sigma$,
and there is a set $A$ of at most $\alpha$ vertices of $G$ such that
$G - A$ can be written as $G_0 \cup G_1 \cup \cdots \cup G_b$ where
$G_0,G_1,\dots,G_b$ are edge-disjoint subgraphs of $G$ and the following
conditions hold:

\begin{itemize}
\item[(N1)] $G_0$ can be embedded in $\Sigma$ up to 3-separations 
with $G''$ being the corresponding reduction of $G_0$.
\item[(N2)] If $1 \leq i < j \leq b$, then $V(G_i) \cap V(G_j) = \emptyset$.
\item[(N3)] $W_i = V(G_0) \cap V(G_i) = V(G'') \cap C_i$
for every $i=1,\dots, b$.
\item[(N4)] For every $i = 1,\dots ,b$, the pair $(G_i,W_i)$ is a vortex
of adhesion less than $\alpha$, where the ordering of $W_i$ is consistent
with the (cyclic) order of these vertices on $C_i$.
\end{itemize}

The vertices in $A$ are called the \DEF{apex vertices} of the
\DEF{$\alpha$-near embedding}.
The subgraph $G_0$ of $G$ is said to be the \DEF{embedded subgraph}
with respect to the $\alpha$-near embedding and the decomposition
$G_0,G_1,\dots,G_b$. The pairs $(G_i,W_i)$, $i=1,\dots,b$, are the
\DEF{vortices} of the $\alpha$-near embedding. The vortex
$(G_i,W_i)$ is said to \DEF{be attached to the cuff $C_i$} of
$\Sigma$ containing $W_i$.

If $G$ is $\alpha$-near-embedded in $S$, let $G_0,G_1,\dots,G_b$ be as above
and let $G''$ be the reduction of $G_0$ that is embedded in $S$. 
If $H$ is an $r$-wall in $G$, we say that $H$ is 
\DEF{captured in the embedded subgraph} $G_0$
of the $\alpha$-near-embedding if $H$ is preserved in the reduction $G''$
and for every separation $G=K\cup L$ of order less than $r$, where 
$G_0\subseteq K$, at least $\tfrac{2}{3}$ of the degree-3 vertices of $H$
lie in $K$.

We shall use the following theorem which is a simplified version of one of the
cornerstones of Robertson and Seymour's theory of graph minors, the Excluded
Minor Theorem, as stated in \cite{RS17}. For a detailed explanation of how the
version in this
paper can be derived from the version in \cite{RS17}, 
see the appendix of \cite{BKMM}.

\begin{theorem}[Excluded Minor Theorem]
\label{th:RS}
For every graph $R$, there is a constant $\alpha$ such that for every
positive integer $w$, there exists a positive integer $r = r(R,\alpha,w)$,
which tends to infinity with $w$ for any fixed $R$ and $\alpha$, such that
every graph $G$ that does not contain an $R$-minor either has tree-width
at most $w$ or contains an $r$-wall\/ $H$ such that $G$
has an $\alpha$-near embedding in some surface $\Sigma$ in which $R$
cannot be embedded, and $H$ is captured in the embedded subgraph of the
near-embedding.
\end{theorem}

We can add the following assumptions about the $r$-wall in Theorem \ref{th:RS}.

\begin{theorem}
\label{thm:wall contained}
It may be assumed that the $r$-wall\/ $H$ in Theorem \ref{th:RS} has
the following properties:
\begin{itemize}
\item[\rm (a)]
$H$ is contained in the reduction $G''$ of the embedded subgraph $G_0$.
\item[\rm (b)]
$H$ is planarly embedded in $\Sigma$, i.e., every cycle in $H$ is
contractible in $\Sigma$ and the outer cycle of $H$ bounds a disk
in $\Sigma$ that contains $H$.
%\item[\rm (c)] Every non-contractible curve in the surface $\Sigma$ 
%intersects $G''$ in at least two vertices.
\item[\rm (c)] We may prespecify any constant $\rho$ and ask that
the face-width of $G''$ be at least $\rho$.
\item[\rm (d)] $G''$ is $3$-connected.
\end{itemize}
\end{theorem}

\begin{proof}
The starting point is Theorem \ref{th:RS}. By making additional elementary
reductions if necessary, we can achieve (d). 
The property (c) is attained as follows.
If the face-width is too small, then there is a set of less than $\rho$
vertices whose removal reduces the genus of the embedding of $G''$.
We can add these vertices in the apex set and repeat the procedure
as long as the face-width is still smaller than $\rho$. The only subtlety here
is that the constant $\alpha$ in Theorem \ref{th:RS} now depends
not only on $R$ but also on $\rho$. See also \cite{KMSurvey}.

After removing the apex set $A$, we are left with an $(r-\a)$-wall in $G-A$. 
By applying Lemma \ref{lem:captured wall}, we may assume that $H$ is contained in the
reduced graph $G''\cup G_1\cup\cdots\cup G_b$. 
The wall $H$ contains a large cylindrical wall $Q$.
Since the vortices have bounded adhesion, they do not have large transactions.
Since the wall is captured in $G''$, edges of many meridians of $Q$ lie in $G''$.
Therefore, we can apply Theorem \ref{th:main} for the near embedding 
of the reduced graph
together with the vortices. This shows that a large cylindrical subwall of $Q$ 
is planarly embedded in the surface. The size $r'$ of this smaller wall still
satisfies the condition that $r' = r'(R,\alpha,w) \to \infty$ as $w$ increases.
\end{proof}

It is worth mentioning that there are other ways to show that a graph with large
enough tree-width that does not contain a fixed graph $R$ as a minor contains a
subgraph that is $\alpha$-near-embedded in some surface $\Sigma$ in which $R$
cannot be embedded, and moreover, there is an $r$-wall planarly embedded in
$\Sigma$ (after reductions taking care of at most 3-separations). Let us describe
two of them:
\begin{itemize}
\item[(A)] {\em Large face-width argument\/}:
One can use property (c) in Theorem \ref{thm:wall contained} that the face-width 
$\rho$ can be made as large as we want if $\alpha=\alpha(R,w,\rho)$ is large
enough. Once we have that, it follows from \cite{RS7} that there is a planarly
embedded $r$-wall, where $r=r(R,\rho)\to\infty$ as $\rho\to\infty$.
While this easy argument is sufficient for most applications, it appears to
be slightly weaker than Theorem \ref{thm:wall contained} since the quantifiers 
change. The difference is that the number of apex vertices
is no longer bounded as a function of $\alpha=\alpha(R)$ but rather as a function
depending on $R$ and $r$, where the upper bound has linear dependence on $r$,
i.e. it is of the form $\beta(R)r$. However, other parameters of the near-embedding
keep being only dependent on $R$. 
\item[(B)] {\em Irrelevant vertex\/}:
The third way of establishing the same result is to go through the proof of
Robertson and Seymour that there is an {\em irrelevant vertex}, i.e. a vertex $v$
such that $G$ has an $R$-minor if and only if $G-v$ has. (Compared to the later,
more abstract parts of the graph minors series of papers, this part is very clean 
and well understood; it could (and should) be explained in a(ny) serious graduate course 
on graph minors.) In that proof, one starts with an arbitrary wall $W$ that is 
large enough. A large wall exists since the tree-width is large. Then one compares
the $W$-bridges attached to $W$. They may give rise to $\le3$-separations, to 
{\em jumps} (paths in bridges whose addition to $W$ yields a nonplanar graph),
{\em crosses} (pairs of disjoint paths attached to the same planar face of $W$
whose addition to $W$ yields a nonplanar graph). If there are many disjoint jumps
or crosses on distinct faces of $W$, one can find an $R$-minor. If there are just 
a few, there is a large planar wall. If there are many of them on the same face,
we get a structure of a vortex with bounded transactions (or else an $R$-minor
can be discovered).
The proof then discusses ways for many jumps and crosses but no large
subset of them being disjoint. One way is to have a small set of vertices whose 
removal destroys most of these jumps and crosses. This gives rise to the apex 
vertices. The final conclusion is that the jumps and crosses can affect only 
a bounded part of the wall, so after the removal of the apex vertices and after
elementary reductions which eliminate $\le3$-separations, 
there is a large subwall $W_0$ such that no jumps or crosses are
involved in it. The ``middle'' vertex in $W_0$ is then shown to be irrelevant.

For our reference, only this {\em planar} wall is needed. By being planar, we mean 
that the rest of the graph is attached only to the outer face of this wall.
Then we define the tangle corresponding to this wall and the proof of the
EMT preserves this tangle while making the modifications yielding to
an $\alpha$-near-embedding.
\end{itemize}

\end{document}